\documentclass{icm2010}
\usepackage{stmaryrd}
\usepackage[all]{xy}

\newtheorem{theorem}{Theorem}[section]
\newtheorem{lemma}[theorem]{Lemma}

\theoremstyle{definition}
\newtheorem{defn}[theorem]{Definition}

\newtheorem{example}[theorem]{Example}
\newtheorem{remark}[theorem]{Remark}

\numberwithin{equation}{theorem}

\def\CC{\mathbb{C}}
\def\DD{\mathbb{D}}

\def\Fp{\mathbb{F}_p}
\def\Qp{\mathbb{Q}_p}
\def\QQ{\mathbb{Q}}
\def\RR{\mathbb{R}}
\def\Zp{\mathbb{Z}_p}
\def\ZZ{\mathbb{Z}}
\def\calE{\mathcal{E}}
\def\calF{\mathcal{F}}
\def\calH{\mathcal{H}}
\def\calM{\mathcal{M}}

\def\calR{\mathcal{R}}
\def\calS{\mathcal{S}}
\def\gothm{\mathfrak{m}}
\def\gotho{\mathfrak{o}}

\def\dual{\vee}

\def\bB{\mathbf{B}}

\DeclareMathOperator{\adm}{adm}
\DeclareMathOperator{\alg}{alg}
\DeclareMathOperator{\an}{an}
\DeclareMathOperator{\Aut}{Aut}
\DeclareMathOperator{\bd}{bd}
\DeclareMathOperator{\Fil}{Fil}
\DeclareMathOperator{\Frac}{Frac}
\DeclareMathOperator{\Gal}{Gal}
\DeclareMathOperator{\GL}{GL}

\DeclareMathOperator{\inte}{int}
\DeclareMathOperator{\Lie}{Lie}
\DeclareMathOperator{\perf}{perf}
\DeclareMathOperator{\rank}{rank}
\DeclareMathOperator{\rig}{rig}

\DeclareMathOperator{\wa}{wa}

\title[Relative $p$-adic Hodge theory and Rapoport-Zink period domains]{Relative $p$-adic Hodge theory and Rapoport-Zink period domains}
\author[Kiran Sridhara Kedlaya]{Kiran Sridhara Kedlaya\thanks{Supported by NSF (CAREER grant DMS-0545904), DARPA (grant HR0011-09-1-0048),
MIT (NEC Fund, Cecil and Ida Green professorship), IAS (NSF grant DMS-0635607, James D. Wolfensohn Fund).}}
\contact[kedlaya@mit.edu]{Department of Mathematics, Massachusetts Institute of Technology, 
77 Massachusetts Avenue, Cambridge, MA 02139, USA}

\begin{document}

\begin{abstract}
As an example of relative $p$-adic Hodge theory, we sketch the 
construction of the universal
admissible filtration of an isocrystal ($\phi$-module) over the completion
of the maximal unramified extension of $\Qp$, together
with the associated universal crystalline local system.
\end{abstract}



\maketitle

\section*{Introduction}

The subject of \emph{$p$-adic Hodge theory} seeks to clarify the relationship between various
cohomology theories (primarily \'etale and de Rham)
associated to algebraic varieties over $p$-adic fields, in much the same way as 
ordinary Hodge theory clarifies the relationship between various cohomology theories (primarily Betti and de Rham)
associated to
complex algebraic varieties. 
Only recently, however, has $p$-adic Hodge theory progressed to the point of dealing comfortably
with \emph{families} of $p$-adic varieties, in the way that one uses variations of Hodge structures to deal with
families of complex varieties.

In this lecture, we illustrate one example of relative $p$-adic Hodge theory: the construction of
the universal admissible filtration on an isocrystal of given Hodge-Tate weights, and the corresponding universal crystalline local system. This problem was originally introduced by Rapoport and
Zink \cite{rapoport-zink}, as part of a generalization of the construction of $p$-adic symmetric spaces
by Drinfel'd \cite{drinfeld}; the relevant spaces in this construction are the moduli of filtered
isocrystals. For $K_0$ an absolutely unramified $p$-adic field with perfect residue field, an \emph{isocrystal} is a
finite-dimensional $K_0$-vector space equipped with an invertible semilinear Frobenius action.
Typical examples are the crystalline cohomology groups of a smooth proper scheme over the residue field of $K_0$.
If the scheme lifts to characteristic 0, one then obtains a \emph{filtered isocrystal} by transferring the
Hodge filtration from de Rham cohomology to crystalline cohomology via the canonical isomorphism.
One can then pass directly from this filtered isocrystal to the
\'etale cohomology of the scheme, by a recipe of Fontaine.

Given an isocrystal, the possible filtrations on it with jumps at particular indices (i.e., with 
prescribed \emph{Hodge-Tate weights}) are naturally parametrized by a partial flag variety.
Of the points of this variety defined over finite extensions of $K_0$, one can identify those which
give rise to Galois representations: by a theorem of Colmez and Fontaine \cite{colmez-fontaine},
they are the ones satisfying a simple linear-algebraic condition called \emph{weak admissibility}
(analogous to the notion of \emph{semistability} in the theory of vector bundles).
Rapoport and Zink conjecture the existence of a rigid analytic subspace of this variety, containing
exactly the weakly admissible points, and admitting a local system specializing at each point to the appropriate
crystalline Galois representation. 
What makes this conjecture subtle is that while the definition of weak admissibility suggests a natural analytic structure
on the set of weakly admissible points, one cannot construct the local system without modifying the Grothendieck topology.
As observed by de Jong
\cite{dejong}, this situation is better understood in Berkovich's language of nonarchimedean analytic spaces:
the space sought by Rapoport-Zink has the same rigid analytic points as the weakly admissible locus, but
is missing some of the nonrigid points.

We construct the Rapoport-Zink space and its associated local system by copying as closely as possible
the corresponding construction in equal characteristic given by Hartl \cite{hartl-equi}. The definition of the
space itself, as suggested by Hartl \cite{hartl-mixed}, is similar in spirit to the definition of weak
admissibility, but it concerns not the original filtered isocrystal but an associated isocrystal over a somewhat
larger ring. In the case of a rigid analytic point, this ring is the \emph{Robba ring} appearing in the 
modern theory of $p$-adic differential equations; it is the ring of germs of power series (over a certain coefficient field)
convergent at the outer boundary of the open unit disc (or more precisely, on some unspecified open annulus with unit outer radius). The classification of isocrystals over the Robba ring, analogous to the classification of rational
Dieudonn\'e modules, was introduced by this author in \cite{kedlaya-annals}. We generalized this classification
\cite{kedlaya-revisited} in a fashion that allows it to be applied to arbitrary Berkovich-theoretic points of the flag variety. Having identified a candidate for the admissible locus, we imitate Berger's alternate
proof of the Colmez-Fontaine theorem \cite{berger-adm} to construct an isocrystal (more precisely a 
$(\phi, \Gamma)$-module) over a relative Robba ring, from which
we construct the desired local system by the usual procedure from $p$-adic Hodge theory (specializing to $p$-th power roots 
of unity and then performing Galois descent).

One of the intended applications of this construction is to the study of period morphisms associated to moduli spaces
of $p$-divisible groups (Barsotti-Tate groups).
Fix a $p$-divisible group $G$ over $\Fp^{\alg}$
of height $h$ and dimension $d$; its rational crystalline Dieudonn\'e module $\DD(G)_{K_0}$ is then
an isocrystal over $K_0 = \Frac(W(\Fp^{\alg}))$. To any complete discrete valuation ring $\gotho_K$
of characteristic $0$ with residue field $\Fp^{\alg}$, and any deformation of $G$
to a $p$-divisible group $\tilde{G}$ over $\gotho_K$, Grothendieck and Messing
\cite{messing} associate an extension
\[
0 \to (\Lie\,\tilde{G}^\dual)_K^\dual \to \DD(G)_{K} \to (\Lie\,\tilde{G})_K \to 0.
\]
This gives $\DD(G)_K$ the structure of a filtered isocrystal with Hodge-Tate weights in $\{0,1\}$,
and determines a $K$-point in the Grassmannian $\calF$ of $(h-d)$-dimensional subspaces of $\DD(G)_{K_0}$.
Grothendieck asked \cite{groth-icm} which points of $\calF$ can occur in this fashion;
Rapoport and Zink proved
\cite[5.16]{rapoport-zink} that all such points belong to the image of a certain
\emph{period morphism} from the generic fibre of the universal deformation space of $G$.
Using results of Faltings, Hartl \cite[Theorem~3.5]{hartl-mixed2} has shown that
his (and our) admissible locus is exactly the image of the Rapoport-Zink period morphism.

The presentation here is based on a lecture series given in January 2010 during the \emph{Trimestre Galoisien}
at Institut Henri Poincar\'e (Paris). The original lecture notes for that series are available online \cite{kedlaya-paris}.
We are currently preparing a more detailed manuscript in collaboration with Ruochuan Liu.

\section{Nonarchimedean analytic spaces}

It is convenient to use Berkovich's language of nonarchimedean analytic spaces. Here is the briefest of synopses of
\cite{berkovich1}.

\begin{defn}
Consider the following conditions on a ring $A$ (always assumed to be 
commutative and unital) and 
a function $\alpha: A \to [0, +\infty)$.
\begin{enumerate}
\item[(a)]
For all $g,h \in A$, we have $\alpha(g-h) \leq \max\{\alpha(g), \alpha(h)\}$.
\item[(b)]
We have $\alpha(0) = 0$.
\item[(b$'$)]
For all $g \in A$, we have $\alpha(g) = 0$ if and only if $g=0$.
\item[(c)]
We have $\alpha(1) = 1$, and for all $g,h \in A$, we have $\alpha(gh) \leq \alpha(g) \alpha(h)$.
\item[(c$'$)]
We have $\alpha(1) = 1$, and for all $g,h \in A$, we have $\alpha(gh) = \alpha(g) \alpha(h)$.
\end{enumerate}
We say $\alpha$ is a \emph{(nonarchimedean) seminorm} if it satisfies (a) and (b), and a
\emph{(nonarchimedean) norm} if it satisfies (a) and (b$'$). We say $\alpha$ is 
\emph{submultiplicative} if it satisfies (c), and \emph{multiplicative} if it satisfies (c$'$).
\end{defn}

\begin{example}
For any ring $A$, the function sending 0 to 0 and every other element of $A$ to 1 is a nonarchimedean
norm, called the \emph{trivial norm}. It is multiplicative if $A$ is an integral domain,
 and submultiplicative otherwise.
\end{example}
\begin{example}
For $R$ a ring equipped with a (sub)multiplicative (semi)norm \mbox{$|\cdot|$}, the \emph{Gauss 
(semi)norm} on $R[T]$ takes
$\sum_i a_i T^i$ to $\max_i\{|a_i|\}$. 
\end{example}

\begin{defn}
Let $A$ be a ring equipped with a submultiplicative norm $|\cdot|$. The \emph{Gel'fand spectrum}
$\calM(A)$ of $A$ is the set of multiplicative seminorms $\alpha$ on $A$ bounded above by $|\cdot|$,
topologized as a closed (hence compact, by Tikhonov's theorem) subspace of the product $\prod_{a \in A} [0, |a|]$.
A subbasis of this topology is given by the sets $\{\alpha \in \calM(A): \alpha(f) \in I\}$
for each $f \in A$ and each open interval $I \subseteq \RR$.
For $\widehat{A}$ the 
separated completion of $A$ with respect to $|\cdot|$, extension by continuity gives a 
natural identification of $\calM(A)$ with $\calM(\widehat{A})$.

For $\alpha \in \calM(A)$, the seminorm $\alpha$ induces a multiplicative norm on the integral domain $A/\alpha^{-1}(0)$,
and hence also on $\Frac(A/\alpha^{-1}(0))$. The completion of this latter field is the \emph{residue field}
of $\alpha$, denoted $\calH(\alpha)$.
\end{defn}

\begin{lemma} \label{L:surjective}
Let $A,B$ be rings equipped with submultiplicative norms $|\cdot|_A$, $|\cdot|_B$.
\begin{enumerate}
\item[(a)]
Let $\phi: A \to B$ be a ring homomorphism for which $|\phi(a)|_B \leq |a|_A$ for all $a \in A$. 
Then $\phi$ induces a continuous 
map $\phi^*: \calM(B) \to \calM(A)$ by restriction.
\item[(b)]
Suppose further that $\phi$ is injective and $A$ admits an orthogonal complement in $B$.
Then $\phi^*$ is surjective.
\end{enumerate}
\end{lemma}
\begin{proof}
Part (a) is clear. For (b), note that for any $\alpha \in \calM(A)$,
$\calH(\alpha)$ admits an orthogonal complement in the completed tensor product
$\calH(\alpha) \widehat{\otimes}_A B$. We may thus apply
\cite[Theorem~1.2.1]{berkovich1} to produce an element $\beta \in
\calM(\calH(\alpha) \widehat{\otimes}_A B)$, whose restriction to $B$ will lie in
the fibre of $\phi^*$ above $\alpha$.
\end{proof}

\begin{example}
Consider the ring $\ZZ$ equipped with the trivial norm. In this case, one may describe
$\calM(\ZZ)$ as the comb
\[
\bigcup_{\mbox{$p$ prime}} \{(c p^{-1}, c p^{-2}): c \in [0,1]\} \subseteq \RR^2,
\]
where
\[
(cp^{-1}, cp^{-2})(p^a m) = (1-c)^a \quad (a,m \in \ZZ; a \geq 0; m \not\equiv 0 \pmod{p}).
\]
In particular, each neighborhood of $(0,0)$ (the trivial norm) contains the complement of some finite union of the given
segments.
\end{example}

\begin{example} \label{exa:closed unit disc}
For $K$ a field complete for a multiplicative norm, equip $K[T]$ with the Gauss norm.
The points of $\calM(K[T])$ (the \emph{closed unit disc} over $K$) have been classified by Berkovich \cite[\S 1]{berkovich1} when $K$ is algebraically closed
(see also \cite[\S 2]{kedlaya-semi4} for the general case).
For instance, for each $z \in K$ with $|z| \leq 1$ and each $r \in [0,1]$, the formula
\begin{equation} \label{eq:taylor formula}
 \alpha_{z,r}(f) = \max_i \left\{r^i \left| \frac{1}{i!} \frac{d^i f}{dT^i}(z) \right| \right\}
\end{equation}
defines a point $\alpha_{z,r} \in \calM(K[T])$;
these comprise all points of $\calM(K[T])$ if and only if $K$ is spherically complete and algebraically closed.
For applications of this space to dynamical systems on the projective line, see 
\cite{baker-rumely}.
\end{example}

\section{Witt vectors}

We will make extensive use of the Witt vectors over a perfect $\Fp$-algebra. Even if these are familiar, some facts about their
Gel'fand spectra may not be.

\begin{defn}
For $R$ a perfect $\Fp$-algebra (i.e., a ring in which $p=0$ and the $p$-th power map is a bijection), 
let $W(R)$ denote the ring of $p$-typical Witt vectors over $R$.
The definition of $W(R)$ may be reconstructed from the following key properties.
\begin{enumerate}
\item[(a)]
The ring $W(R)$ is $p$-adically complete and separated, and $W(R)/(p) \cong R$.
\item[(b)]
For each $r \in R$, there is a unique lift $[r]$ of $r$ to $W(R)$ (the \emph{Teichm\"uller lift})
having $p^n$-th roots for all positive integers $n$.
\item[(c)]
Each $x \in W(R)$ admits a unique representation
$\sum_{i=0}^\infty p^i [x_i]$ with $x_i \in R$.
\end{enumerate}
Since the construction of $W(R)$ is functorial in $R$, $W(R)$ also carries an automorphism $\phi$ (the
\emph{Witt vector Frobenius}) lifting the $p$-power Frobenius map on $R$.
We equip $W(R)$ with the \emph{normalized $p$-adic norm}, that is, the
norm of a nonzero element
$\sum_{i=0}^\infty p^i [x_i]$ equals $p^{-j}$ for $j$ the smallest index with $x_j \neq 0$.
\end{defn}

If $R$ carries a submultiplicative norm and $R[T]$ carries the corresponding Gauss norm, 
we obtain a map $\lambda: \calM(R) \to \calM(R[T])$ taking each seminorm on $R$ to its
Gauss extension, and a map $\mu: \calM(R[T]) \to \calM(R)$ induced by the 
inclusion $R \to R[T]$. For $R$ a perfect $\Fp$-algebra, we have similar maps between $\calM(R)$
and $\calM(W(R))$, with the role of the inclusion $R \to R[T]$ played by the multiplicative
(but not additive) Teichm\"uller map.

\begin{lemma} \label{L:lambda}
Equip $R$ with the trivial norm.
For $\alpha \in \calM(R)$, the function
$\lambda(\alpha): W(R) \to [0,1]$ given by
\[
\lambda(\alpha) \left( \sum_{i=0}^\infty p^i [x_i]
\right) = \max_i \{ p^{-i} \alpha(x_i) \}.
\]
is a multiplicative seminorm bounded by the $p$-adic norm, and so belongs to $\calM(W(R))$.
\end{lemma}
\begin{proof}
Let $x = \sum_{i=0}^\infty p^i [x_i],
y = \sum_{i=0}^\infty p^i [y_i]$ be two general elements of $W(R)$.
If we write $x+y = \sum_{i=0}^\infty p^i [z_i]$, then each $z_i$
is a polynomial in $x_j^{p^{j-i}},y_j^{p^{j-i}}$ for $j=0,\dots,i$,
which is homogeneous of degree 1 for the weighting in which $x_j,y_j$ have degree 1.
It follows that $\lambda(\alpha)(x+y) \leq \max\{\lambda(\alpha)(x),
\lambda(\alpha)(y)\}$, so $\lambda(\alpha)$ is a seminorm. This in turn implies that
\begin{align*}
\lambda(\alpha)(xy) &\leq \max_{i,j} \{\lambda(\alpha)(p^i[x_i] p^j [y_j])\} \\
&\leq \lambda(\alpha)(x) \lambda(\alpha)(y),
\end{align*}
so $\lambda(\alpha)$ is submultiplicative.
To check multiplicativity, we may safely assume
$\lambda(\alpha)(x), \lambda(\alpha)(y) > 0$. Choose the minimal indices
$j,k$ for which $\lambda(\alpha)(p^j [x_j])$, 
$\lambda(\alpha)(p^k [y_k])$ attain their maximal values. For
\[
x' = \sum_{i=j}^{\infty} p^i [x_i],
\qquad
y' = \sum_{i=k}^{\infty} p^i [y_i],
\]
on one hand we have $\lambda(\alpha)(x-x') < \lambda(\alpha)(x)$, $\lambda(\alpha)(y-y') < \lambda(\alpha)(y)$.
Since $\lambda(\alpha)$ is a submultiplicative seminorm,
we get that $\lambda(\alpha)(xy) = \lambda(\alpha)(x'y')$.
On the other hand, we may write $x'y' = \sum_{i=j+k}^\infty p^i [z_i]$ with
$z_{j+k} = x_j y_k$.
Therefore $\lambda(\alpha)(x'y') \geq \lambda(\alpha)(x) \lambda(\alpha)(y)$. Putting everything together,
we deduce that $\lambda(\alpha)$ is multiplicative.
\end{proof}

\begin{lemma}\label{L:retract}
Equip $W(R)$ with the $p$-adic norm.
For $\beta \in \calM(W(R))$, the function $\mu(\beta): R \to [0,1]$
given by
\[
\mu(\beta)(x) = \beta([x])
\]
is a multiplicative seminorm bounded by the trivial norm, and so belongs to $\calM(R)$.
\end{lemma}\begin{proof}
Given $x_0, y_0 \in R$, choose any $x,y \in W(R)$ lifting them.
For
$(z_0, z) = (x_0,x), (y_0,y), (x_0+y_0, x+y)$,
for any $\epsilon >0$,
for $n$ sufficiently large (depending on $z,z_0,\epsilon$), we have
$\max\{\epsilon,\mu(\beta)(z_0)\} = \max\{\epsilon,\beta(\phi^{-n}(z))^{p^n}\}$ because $\phi^{-n}(z^{p^n})$ converges
$p$-adically to $[z_0]$.
Since $\beta$ is a multiplicative seminorm, we deduce the same for $\mu(\beta)$. 
\end{proof}

\begin{theorem} \label{T:lifting}
Equip $R$ with the trivial norm and $W(R)$ with the $p$-adic norm.
Then the functions $\lambda: \calM(R) \to \calM(W(R)), \mu: \calM(W(R)) \to \calM(R)$ are continuous.
Moreover, for any $\alpha \in \calM(R), \beta \in \calM(W(R))$, we have $(\mu \circ \lambda)(\alpha) = \alpha$
and $(\lambda \circ \mu)(\beta) \geq \beta$. (The latter means that for any $x \in W(R)$,
$(\lambda \circ \mu)(\beta)(x) \geq \beta(x)$.)
\end{theorem}
\begin{proof}
For $x = \sum_{i=0}^\infty p^i [x_i] \in W(R)$ and $\epsilon > 0$, choose $j > 0$ for which $p^{-j} < \epsilon$; then
$\lambda(\alpha)(p^i [x_i]) < \epsilon$ for all $\alpha \in \calM(R)$ and all $i \geq j$. We thus have
\begin{align*}
\{\alpha \in \calM(R): \lambda(\alpha)(x) > \epsilon\} &= \bigcup_{i=0}^{j-1}
\{\alpha \in \calM(R): \alpha(x_i) > p^i \epsilon \} \\
\{\alpha \in \calM(R): \lambda(\alpha)(x) < \epsilon\} &= \bigcap_{i=0}^{j-1}
\{\alpha \in \calM(R): \alpha(x_i) < p^i \epsilon \},
\end{align*}
and the sets on the right are open. It follows that $\lambda$ is continuous.

For $x_0 \in R$ and $\epsilon > 0$, we have
\begin{align*}
\{\beta \in \calM(W(R)): \mu(\beta)(x_0) > \epsilon\}
&=
\{\beta \in \calM(W(R)): \beta([x_0]) > \epsilon\} \\
\{\beta \in \calM(W(R)): \mu(\beta)(x_0) < \epsilon\}
&=
\{\beta \in \calM(W(R)): \beta([x_0]) < \epsilon\},
\end{align*}
and the sets on the right are open. It follows that $\mu$ is continuous.

The equality $(\mu \circ \lambda)(\alpha) = \alpha$ is evident from the definitions. The inequality
$(\lambda \circ \mu)(\beta) \geq \beta$ follows from the definition of $\lambda$ and the observation
that $(\lambda \circ \mu)(\beta)([x_0]) = \beta([x_0])$ for any $x_0 \in R$.
\end{proof}

\begin{example}
Here is a simple example to illustrate that $\lambda \circ \mu$ need not be the identity map.
Put $R = \cup_{n=1}^\infty \Fp[X^{p^{-n}}]$, so that $W(R)$ is isomorphic to the $p$-adic completion of
$\cup_{n=1}^\infty \Zp[[X]^{p^{-n}}]$. The ring $W(R)/([X]-p)$ is isomorphic to the completion of
$\cup_{n=1}^\infty \Zp[p^{p^{-n}}]$ for the unique multiplicative extension of the $p$-adic norm; let 
$\beta \in \calM(W(R))$ be the induced seminorm.

Note that $\mu(\beta)(X) = \beta([X]) = p^{-1}$ and that $\mu(\beta)(y) = 1$ for $y \in \Fp^\times$.
These imply that $\mu(\beta)(y) \leq p^{-p^{-n}}$ whenever $y \in \Fp[X^{p^{-n}}]$ is divisible by $X^{p^{-n}}$, so
$\mu(\beta)(y) = 1$ whenever $y \in \Fp^\times + X^{p^{-n}} \Fp[X^{p^{-n}}]$. We conclude that for
$y \in R$, $\mu(\beta)(y)$ equals the $X$-adic norm of $y$ with the normalization $\mu(\beta)(X) = p^{-1}$.
In particular, we have a strict inequality $(\lambda \circ \mu)(\beta) > \beta$.
\end{example}

\begin{remark}
There is a strong analogy between the geometry of the fibres of $\mu$ and the geometry of closed discs 
(see Example~\ref{exa:closed unit disc}).
This suggests the possibility of constructing 
a homotopy between the  map $\lambda \circ \mu$ on $\calM(W(R))$ and the identity map,
which acts within fibres of $\mu$ and fixes the image of $\lambda \circ \mu$; such a construction would imply that any subset
of $\calM(R)$ has the same homotopy type as its inverse image under $\mu$. 
Such a homotopy does in fact exist; see \cite{kedlaya-wittgeometry}.
\end{remark}

\section{Filtered isocrystals and weak admissibility}

To simplify the exposition, we introduce filtered isocrystals only for the group $\GL_n$. 
One can generalize to an arbitrary reductive Lie group (see \cite[Chapter~1]{rapoport-zink} for the setup),
but the general results  can be deduced from the $\GL_n$ case.

\begin{defn}
Put $K_0 = \Frac W(\Fp^{\alg})$. 
An \emph{isocrystal} over $K_0$ is a finite-dimensional $K_0$-vector
space equipped with an invertible semilinear action of the Witt vector Frobenius $\phi$.
For $D$ a nonzero isocrystal over $K_0$, the \emph{degree} $\deg(D)$ of $D$ is the $p$-adic valuation
of the determinant of the matrix via which $\phi$ acts on some (and hence any) basis of $D$. 
The \emph{slope} of $D$ is the ratio $\mu(D) = \deg(D)/\rank(D)$.
\end{defn}

\begin{defn}
Let $K$ be a complete extension of $K_0$ (not necessarily discretely valued). A \emph{filtered isocrystal}
over $K$ consists of an isocrystal $D$ over $K_0$ equipped with an exhaustive decreasing filtration
$\{\Fil^i D_K\}_{i \in \ZZ}$ on $D_K = D \otimes_{K_0} K$. 
The \emph{Hodge-Tate weights} are then defined as the multiset
containing $i \in \ZZ$ with multiplicity $\dim_K (\Fil^i D_K)/(\Fil^{i+1} D_K)$. 

For $D$ an isocrystal over $K_0$ and $H$ a finite multiset of integers, let $\calF_{D,H}$ be the partial flag variety
parametrizing exhaustive decreasing filtrations on $D$ with Hodge-Tate weights $H$. Let $\calF^{\an}_{D,H}$
be the analytification of $\calF_{D,H}$ in the sense of \cite[Theorem~3.4.1]{berkovich1}.
\end{defn}
Beware that our notion of filtered isocrystals is slightly nonstandard; 
for instance, if $K$ is a finite extension of $K_0$, one would normally
replace $K_0$ by its maximal unramified extension within $K$. Since our ultimate goal is to
fix an isocrystal structure and vary the filtration, this discrepancy is not so harmful.

\begin{lemma} \label{L:cover by discs}
Equip $K_0[T_1^{\pm}, \dots, T_d^{\pm}]$ with the Gauss norm. Then
$\calF^{\an}_{D,H}$ is covered by finitely many copies of
$\calM(K_0[T_1^{\pm}, \dots, T_d^{\pm}])$.
\end{lemma}
\begin{proof}
We first observe that the closed unit disc $\calM(K_0[T])$ is covered by
$\calM(K_0[T^\pm])$ and $\calM(K_0[(T-1)^{\pm}])$. It follows that 
$\calM(K_0[T_1, \dots, T_d])$ is covered by finitely many copies of
$\calM(K_0[T_1^{\pm}, \dots, T_d^{\pm}])$. 

Let $\mathfrak{F}_{D,H}$ be the partial flag variety over $W(\Fp^{\alg})$ with generic fibre $\calF_{D,H}$.
Then $(\mathfrak{F}_{D,H})_{\Fp^{\alg}}$ is a partial flag variety over $\Fp^{\alg}$ of dimension $d$,
and so can be covered by 
finitely many $d$-dimensional affine spaces (e.g., using Pl\"ucker coordinates). 
Lifting such a covering to the $p$-adic formal completion of $\mathfrak{F}_{D,H}$,
then taking (Berkovich) analytic generic fibres, yields a covering of
$\calF^{\an}_{D,H}$ by finitely many copies of $\calM(K_0[T_1, \dots, T_d])$. By the previous paragraph,
this implies the desired result.
\end{proof}

\begin{defn}
Let $D$ be a filtered isocrystal over some complete extension of $K_0$.
Define
$t_N(D) = \deg(D)$, and let $t_H(D)$ be the sum of the Hodge-Tate weights of $D$.
We say $D$ is \emph{weakly admissible} if the following conditions hold.
\begin{enumerate}
\item[(a)]
We have $t_N(D) = t_H(D)$.
\item[(b)]
For any subisocrystal ($\phi$-stable subspace) $D'$ of $D$  equipped with the induced filtration, 
$t_N(D') \geq t_H(D')$.
\end{enumerate}
\end{defn}

Weak admissibility is an open condition, in the following sense.
\begin{theorem}
Let $\calF^{\wa}_{D,H}$ be the set of $\alpha \in \calF^{\an}_{D,H}$ for which $D$ becomes weakly admissible
when equipped with the filtration on $D_{\calH(\alpha)}$
induced by the universal filtration over $\calF_{D,H}$.
Then $\calF^{\wa}_{D,H}$ is open in $\calF^{\an}_{D,H}$.
\end{theorem}
\begin{proof}
See \cite[Proposition~1.36]{rapoport-zink}.
\end{proof}

\begin{defn}
The tensor product of two filtrations $\Fil_1^\cdot, \Fil_2^\cdot$ is given by
\[
(\Fil_1 \otimes \Fil_2)^k = \sum_{i+j=k} \Fil_1^i \otimes \Fil_2^j.
\]
It is true but not immediate that the tensor product of two weakly admissible filtered isocrystals is weakly admissible; this 
was proved by Faltings \cite{faltings85} and
Totaro \cite{totaro}, using ideas from geometric invariant theory.
It also follows \emph{a posteriori} from describing weak admissibility in terms of
Galois representations, or in terms of isocrystals over the Robba ring (Theorem~\ref{T:Berger}).
\end{defn}

\section{Admissibility at rigid analytic points}

For the rest of the paper, fix an isocrystal $D$ over $K_0$ and a finite multiset $H$ of integers.
In this section, we follow Berger's proof of the Colmez-Fontaine theorem \cite{berger-adm}, forging a link
between filtered isocrystals and crystalline Galois representations via Frobenius modules over the Robba ring.
We will use this link as the basis
for our definition of the admissible locus of $\calF_{D,H}^{\an}$. (One could use instead Kisin's variant of
Berger's method \cite{kisin}; see Remark~\ref{R:Kisin}.)

\begin{defn}
Fix once and for all a completed algebraic closure $\CC_{K_0}$ of $K_0$
and a sequence $\epsilon = (\epsilon_0,\epsilon_1,\dots)$ in $\CC_{K_0}$
in which $\epsilon_i$ is a primitive $p^i$-th root of 1 and $\epsilon_{i+1}^p = \epsilon_i$.
(This is analogous to fixing a choice of $\sqrt{-1}$ in the complex numbers, in order to specify
orientations.) Write $K_0(\epsilon)$ as shorthand for
$\cup_{n=1}^\infty K_0(\epsilon_n)$.
\end{defn}

\begin{defn}
For $r>0$, define the valuation $v_r$ on $K_0[\pi^{\pm}]$ by the formula
\[
v_r\left(\sum_i a_i \pi^i \right) = \min_i \{v_p(a_i) + ir\}.
\]
Let $\calR^r$ be the Fr\'echet completion of $K_0[\pi^{\pm}]$ for the valuations
$v_s$ for $s \in (0,r]$; we may interpret $\calR^r$ as the ring of formal Laurent series over $K_0$
convergent in the range $v_p(\pi) \in (0,r]$. Put $\calR = \cup_{r>0} \calR^r$ (the \emph{Robba ring} over $K_0$) and
\[
t = \log (1 + \pi) = \sum_{i=1}^\infty \frac{(-1)^{i-1}}{i} \pi^i \in \calR.
\]
Let $\calR^{\bd}$ be the subring of $\calR$ consisting of series with bounded coefficients;
then $\calR^{\bd}$ is a henselian (but not complete) discretely valued field for the $p$-adic valuation,
with residue field $\Fp^{\alg}((\overline{\pi}))$. Let $\calR^{\inte}$ be the valuation subring of $\calR^{\bd}$.

Let $\phi: \calR \to \calR$ be the map
\[
\phi\left(\sum_i a_i \pi^i \right) = \sum_i \phi(a_i) ((1 + \pi)^p - 1)^i;
\]
note that $\phi(t) = pt$.
Define also an action of the group $\Gamma = \Zp^\times$ on $\calR$ by the formula
\[
\gamma\left(\sum_i a_i \pi^i \right) = \sum_i a_i ((1 + \pi)^\gamma - 1)^i;
\]
note that $\gamma(t) = \gamma t$, and that the action of $\Gamma$ commutes with $\phi$.
\end{defn}

\begin{defn} \label{D:base change}
For $r>0$ and $n$ a positive integer such that $r \geq 1/(p^{n-1}(p-1))$,
the series belonging to $\calR^r$ converge at $\epsilon_n-1$.
Moreover, $t$ vanishes to order 1 at $\epsilon_n - 1$.
We thus have a well-defined homomorphism $\theta_n: \calR^r \to K_0(\epsilon_n) \llbracket t \rrbracket$ with dense image.
We also have a commutative diagram
\begin{equation} \label{eq:base change}
\xymatrix{
\calR^r \ar^{\phi}[r] \ar^{\theta_n}[d] &
\calR^{r/p} \ar^{\theta_{n+1}}[d] \\
K_0(\epsilon_n) \llbracket t \rrbracket \ar[r] &
K_0(\epsilon_{n+1}) \llbracket t \rrbracket
}
\end{equation}
whenever $r \leq p/(p-1)$, in which the bottom horizontal arrow acts on $K_0$ via $\phi$, fixes $\epsilon_n$, and carries $t$ to $pt$.

Let $S$ be a finite \'etale algebra over $\calR^{\inte}$. Choose some $r$ for which $S$
can be represented as the base extension of a finite \'etale algebra $S_r$ over $\calR^{\inte} \cap \calR^r$.
For $n$ sufficiently large, we can then form $K_{S,n} = S_r \otimes_{\theta_n} K_0(\epsilon)$.
View the right side as a $K_0(\epsilon)$-algebra via the unique extension of $\phi^n$ fixing all of the $\epsilon_i$; then 
by \eqref{eq:base change}, $\phi$ induces an isomorphism $K_{S,n} \cong K_{S,n+1}$
of $K_0(\epsilon)$-algebras. We thus obtain functorially
from $S$ a finite \'etale algebra $K_S$ over $K_0(\epsilon)$.
\end{defn}

\begin{theorem} \label{T:finite etale}
The functor $S \mapsto K_S$ is an equivalence of categories.
\end{theorem}
\begin{proof}
This is typically deduced from Fontaine's theory of $(\phi, \Gamma)$-modules \cite{fontaine-festschrift},
but it can also be obtained as follows. For full faithfulness, it suffices to check that if $S$ is a field, then so is $K_S$. For this, choose a uniformizer of the residue field of $S$,
and write down the minimal polynomial over $\Fp^{\alg}\llbracket \overline{\pi} \rrbracket$. This polynomial is Eisenstein,
so when we lift to $\calR^{\inte}$ and tensor with $\theta_n$, for $n$ large we get an Eisenstein (and hence irreducible) polynomial over $K_0(\epsilon_n)$. Hence $K_S$ is a field.

For essential surjectivity, it suffices to check that every field $L$ finite over $K_0(\epsilon)$ occurs as a $K_S$. For some positive integer $n$, we can write $L = L_n(\epsilon)$ for some finite extension
$L_n$ of $K_0(\epsilon_n)$ with $[L_n:K_0(\epsilon_n)] = [L:K_0(\epsilon)]$. 
Fix some $r \geq 1/(p^{n-1}(p-1))$.
Note that any $a \in K_0(\epsilon_n)$ with positive $p$-adic valuation can be lifted to $\tilde{a} 
\in \calR^{\inte} \cap \calR^r$ having positive $p$-adic valuation, while any nonzero  $a \in K_0(\epsilon_n)$
can be lifted to $\tilde{a} \in \calR^{\inte} \cap \calR^r$ having zero $p$-adic valuation.

If $L_n$ is tamely ramified over $K_0(\epsilon_n)$, then $L_n = K_0(\epsilon_n)(a^{1/m})$
for some positive integer $m$ not divisible by $p$ and some $a \in K_0(\epsilon_n)$ of positive valuation.
In this case, put $\tilde{P}(T) = T^m - a$ for some lift
$\tilde{a} \in \calR^{\inte} \cap \calR^r$ of $\phi^{-n}(a)$ having zero $p$-adic valuation.

If $L_n$ is wildly ramified over $K_0(\epsilon_n)$, choose $\alpha \in L_n$ of positive valuation
with $L_n = K_0(\epsilon_n)(\alpha)$ and $\mathrm{Trace}_{L_n/K_0(\epsilon_n)}(\alpha) \neq 0$. (We can enforce this last
condition by replacing $\alpha$ by $\alpha+p$ if needed.) Let $P(T) = T^m + \sum_{i=0}^{m-1} a_i T^i$ be the minimal polynomial of $\alpha$ over $K_0(\epsilon_n)$, so that $m$ is divisible by $p$ and $a_0, a_{m-1} \neq 0$.
Put $\tilde{P}(T) = T^m + \sum_{i=0}^{m-1} \tilde{a}_i T^i$ where each
$\tilde{a}_i \in \calR^{\inte} \cap \calR^r$ is a lift of $\phi^{-n}(a_i)$, and $\tilde{a}_i$ has zero $p$-adic valuation
if and only if $i \in \{0,m-1\}$.

In both cases, we obtain a residually separable polynomial $\tilde{P}(T)$ over $\calR^{\inte}$ for which
$S= \calR^{\inte}[T]/(\tilde{P}(T))$ satisfies $K_S = L$.
\end{proof}

\begin{defn}
For $L$ a finite extension of $K_0$, let $\bB^\dagger_{\rig,L}$ be the finite extension of $\calR$
obtained by starting with $L(\epsilon)$, producing a finite \'etale extension $S$ of $\calR^{\inte}$ with $K_S = L(\epsilon)$
using Theorem~\ref{T:finite etale},
and base-extending to $\calR$. This ring carries unique extensions of the actions of
$\phi$ and $\Gamma$; it admits a ring isomorphism to $\calR$, but not in a canonical way (and not respecting the 
actions of $\phi$ or $\Gamma$).

Although $\bB^\dagger_{\rig,L}$
does not carry a $p$-adic valuation, the units in this ring are series with bounded coefficients (by
analysis of Newton polygons), and so do have well-defined $p$-adic valuations.
If we then define an \emph{isocrystal} over
$\bB^\dagger_{\rig,L}$ to be a finite free module
equipped with a semilinear action of $\phi$ acting on some (hence any)
basis via an invertible matrix,
it again makes sense to define \emph{degree} and \emph{slope}.
An isocrystal is \emph{\'etale} if it admits a basis via which $\phi$
acts via an invertible matrix over $S$.
\end{defn}

Beware that $L$ does not embed into $\bB^\dagger_{\rig,L}$, as indicated by the following lemma.
\begin{lemma} \label{L:int closed}
For any finite extension $L$ of $K_0$, $K_0$ is integrally closed in 
$\Frac(\bB^\dagger_{\rig,L})$.
\end{lemma}
\begin{proof}
Suppose $r,s \in \bB^\dagger_{\rig,L}$ and $f=r/s$ is integral over $K_0$.
Then for any maximal ideal $\gothm$ of $\calR$ away from the support of $s$, the image of $f$ in
$\bB^\dagger_{\rig,L} \otimes_{\calR} \calR/\gothm$ must be a root of a fixed polynomial
over $K_0$. This implies that all of these images are bounded in norm, so in fact
$f \in S[p^{-1}]$. In particular, $f$ generates a finite \emph{unramified} extension of $K_0$.
Since $K_0$ has algebraically closed residue field, this forces $f \in K_0$.
\end{proof}

\begin{lemma} \label{L:frac Robba}
For any finite extension $L$ of $K_0$ and any open subgroup $U$ of $\Gamma$,
we have $(\Frac({\bB}^\dagger_{\rig,L}))^U = K_0$.
\end{lemma}
\begin{proof}
Suppose first that $f \in \Frac(\calR)$ is $\Gamma$-invariant. 
Choose $r>0$ for which $f \in \Frac(\calR^r)$. For some large $n$,
we
can embed $\Frac(\calR^r)$ into $K_0(\epsilon_n)((t))$ as in Definition~\ref{D:base change}.
This action is $\Gamma$-equivariant for the action on $K_0(\epsilon_n)$ via the cyclotomic character
(i.e., with $\gamma(\epsilon_n) = \epsilon_n^{\gamma}$)
and the substitution $t \mapsto \gamma t$. It is evident that the fixed subring of $K_0(\epsilon_n)((t))$
under this action is precisely $K_0$, whence the claim.

In the general case, if $f \in \Frac({\bB}^\dagger_{\rig,L})$ is $U$-invariant, then it is integral
over $\Frac(\calR)^U$ and hence over $\Frac(\calR)^\Gamma = K_0$. By Lemma~\ref{L:int closed},
this forces $f \in K_0$.
\end{proof}

\begin{theorem} \label{T:slope}
Let $L$ be a finite extension of $K_0$.
An isocrystal $M$ over ${\bB}^\dagger_{\rig,L}$ is \'etale if and only if the following conditions
hold.
\begin{enumerate}
\item[(a)]
We have $\deg(M) = 0$.
\item[(b)]
For any nonzero subisocrystal $M'$ of $M$, $\deg(M') \geq 0$.
\end{enumerate}
\end{theorem}
\begin{proof}
This is a consequence of slope theory for isocrystals over the Robba ring, as introduced in 
\cite{kedlaya-annals}. See \cite[Theorem~1.7.1]{kedlaya-relative} for a simplified presentation. 
(For less detailed expositions, see also \cite{colmez-bourbaki}
and \cite[Chapter~16]{kedlaya-course}.)
\end{proof}

\begin{defn}
Let $M$ be an isocrystal over ${\bB}^\dagger_{\rig,L}$. 
For $n$ sufficiently large, we may base-extend along $\theta_n$ to produce
a module $M^{(n)}$ over $(K_0(\epsilon_n) \otimes_{\phi^n,K_0} L) ((t))$.
The construction is not canonical (it depends on the choice of a model of $M$ over some $\calR^r$),
but any two such constructions give the same answers for $n$ large. Hence any assertion only concerning the $M^{(n)}$ 
for $n$ large is well-posed.
\end{defn}

The following result of Berger \cite[\S III]{berger-adm} makes the link between weak admissibility and
slopes of isocrystals over the Robba ring.
\begin{theorem}[Berger] \label{T:Berger}
Let $L$ be a finite extension of $K_0$.
Let $(D, \Fil^\cdot D)$ be a filtered isocrystal over $L$, and view
$M = D \otimes_{K_0} {\bB}^\dagger_{\rig,L}$ as an isocrystal over ${\bB}^\dagger_{\rig,L}$.
\begin{enumerate}
\item[(a)]
There exists an isocrystal $M'$ over ${\bB}^\dagger_{\rig,L}$ and a $\phi$-equivariant isomorphism
$M[t^{-1}] \cong M'[t^{-1}]$ of modules over ${\bB}^\dagger_{\rig,L}[t^{-1}]$, via which 
for $n$ sufficiently large, the $t$-adic filtration on $(M')^{(n)}$ coincides with the
filtration on $M^{(n)}$ obtained by tensoring the $t$-adic filtration with the one provided by $D$.
\item[(b)]
The isocrystal $M'$ is \'etale if and only if $(D, \Fil^{\cdot} D)$ is weakly admissible.
\end{enumerate}
\end{theorem}
\begin{proof}
For (a), Berger gives an algebraic construction of $M'$ \cite[\S III.1]{berger-adm}.
An alternative geometric approach 
is to construct $M'$ as an object in the category of coherent locally free sheaves
on an open annulus of outer radius 1.
One then uses the fact (essentially due to Lazard)
that on an annulus over a complete discretely valued field, any coherent locally free sheaf is generated by finitely
many global sections \cite[Theorem~3.14]{kedlaya-mono-over}.

For (b), observe that $\deg(M') = t_N(D) - t_H(D)$.
Hence condition (a) of weak admissibility holds if and only if $\deg(M') = 0$.
If condition (b) fails for some $D'$, then $D' \otimes_{K_0} {\bB}^\dagger_{\rig,L}$ is 
a subisocrystal of $M'$ of negative degree, so 
by Theorem~\ref{T:slope}, $M'$ cannot be \'etale. Conversely, if $M'$ fails to be \'etale,
then by Theorem~\ref{T:slope} it has a subisocrystal $N'$ of negative slope, which we may assume to be saturated (otherwise
its saturation has even smaller degree). There is a unique saturated subisocrystal $N$ of $M$ for which
the isomorphism $M[t^{-1}] \cong M'[t^{-1}]$ induces an
isomorphism $N[t^{-1}] \cong N'[t^{-1}]$.

To deduce that $D$ is not weakly admissible, one must check that $N$ arises from a subisocrystal of $D$.
Put $D' = D \cap N$, so that the natural map $D/D' \otimes_{K_0} {\bB}^\dagger_{\rig,L} \to M/N$ is surjective.
We check that this map is also injective. Suppose the contrary,
and choose an element $\sum_{i=1}^n d_i \otimes r_i$ mapping to zero in $M/N$ with $n$ minimal.
Then for each $j \in \{1,\dots,n\}$ and each $\gamma \in \Gamma$, 
$\sum_{i \neq j} d_i \otimes (r_i \gamma(r_j) - r_j \gamma(r_i))$ maps to zero
in $M/N$. By the minimality of $n$, we have $r_i \gamma(r_j) = r_j \gamma(r_i)$ for all $i,j \in \{1,\dots,n\}$
and all $\gamma \in \Gamma$. By Lemma~\ref{L:frac Robba}, $r_i/r_j \in K_0^\times$ for all $i,j \in \{1,\dots,n\}$,
which forces the $d_i$ to be linearly dependent over $K_0$. But then one can rewrite $d_1$ in terms of $d_2,\dots,d_n$
to reduce the value of $n$, a contradiction.

We now conclude that $\dim_{K_0} D' = 
\rank N$, so $N = D' \otimes_{K_0} {\bB}^\dagger_{\rig,L}$. 
Since $t_N(D') - t_H(D') = \deg(N') < 0$, we conclude that
$D$ is not weakly admissible. (See \cite[Corollaire~III.2.5]{berger-adm}
for a similar argument using the Lie algebra of $\Gamma$.)
\end{proof}

To get from \'etale isocrystals over the Robba ring to Galois representations, 
we proceed as in Definition~\ref{D:base change}.
\begin{defn} \label{D:to Galois}
Let $L$ be a finite extension of $K_0$.
Let $M'$ be an \'etale isocrystal over ${\bB}^\dagger_{\rig,L}$.
Choose a basis of $M$ on which $\phi$ acts via an invertible matrix 
over the valuation subring $\gotho$ of ${\bB}^\dagger_{\rig,L}$.
Let $N$ be the $\gotho$-span of this basis.
For each positive integer $n$, 
we obtain
a connected finite \'etale Galois algebra $S_n$ over $\gotho$
for which $N \otimes_{\gotho} S_n/(p^n)$ admits a $\phi$-invariant basis.
For $n$ large, we can base-extend $S_n$ via $\theta_n$ to obtain a 
(not necessarily connected)
finite \'etale Galois algebra over $K_0(\epsilon) \otimes_{K_0} L$
(which itself may not be connected).
From the Galois action on $\phi$-invariant elements of $N \otimes_{\gotho} S_n/(p^n)$,
we obtain a Galois representation $G_{L(\epsilon)} \to \GL_m(\Zp/p^n \Zp)$ for $m= \rank M'$.
(The coefficient ring is $\Zp/p^n \Zp$ because that is the $\phi$-invariant subring of $\gotho/p^n \gotho$.)
Taking the inverse limit, we obtain a Galois representation $G_{L(\epsilon)} \to \GL_m(\Zp)$;
the resulting representation $G_{L(\epsilon)} \to \GL_m(\Qp)$ does not depend on the original
choice of a basis of $M$.

Let $\Gamma$ act on $K_0(\epsilon)$ via the cyclotomic character, and let $\Gamma_L$ be the open subgroup fixing
$L \cap K_0(\epsilon)$.
Via $\theta_n$, we have
\[
\Gamma_L \cong \Gal(K_0(\epsilon)/(L \cap K_0(\epsilon))) = \Gal(L(\epsilon)/L),
\]
while for any finite Galois extension $L'$ of $L$, 
\[
\Aut({\bB}^\dagger_{\rig,L'}/\bB^\dagger_{\rig,L})
\cong \Gal(L'(\epsilon)/L(\epsilon)).
\]
Taking the semidirect product yields an action of $\Gal(L'(\epsilon)/L)$
on ${\bB}^\dagger_{\rig,L'}$ commuting with $\phi$.
Similarly, if $M'$ comes with an action of $\Gamma_L$ commuting with the $\phi$-action,
then combining this 
$\Gamma_L$-action with the action of $\Aut(S/{\bB}^\dagger_{\rig,L})$
provides descent data on the Galois representation previously constructed, yielding
a Galois representation $G_L \to \GL_m(\QQ_p)$.

In the case arising from Theorem~\ref{T:Berger}, we obtain a $\Gamma_L$-action by 
defining such an action on $M$ fixing $D$.
Berger shows \cite{berger-adm} that the resulting Galois representation
is \emph{crystalline} in Fontaine's sense of having a full set of
periods within the crystalline period ring 
$\bB_{\mathrm{crys}}$. Moreover, the passage between the filtered isocrystal and the Galois representation is compatible
with Fontaine's construction of the \emph{mysterious functor}. That is, if one starts with the de Rham cohomology
of a smooth proper scheme $X$ over $\gotho_L$, viewed as a filtered isocrystal (using the Hodge filtration plus the Frobenius
action on crystalline cohomology), the resulting Galois representation may be 
identified with the $p$-adic \'etale cohomology of
$X$ over $L^{\alg}$.
\end{defn}

\section{Admissibility at general analytic points}

Over a finite extension of $K_0$, we have now interpreted weak admissibility of filtered isocrystals in terms of isocrystals over Robba rings, and given a mechanism for passing from such isocrystals to
Galois representations. 
As noted by Hartl \cite{hartl-mixed}, it seems to be difficult to give a direct analogue of Berger's construction
over a general extension of $K_0$.
In the case of Hodge-Tate weights in $\{0,1\}$, 
Hartl has given an analogue of the notion of admissibility, using a
\emph{field of norms} construction in the manner
of Fontaine and Wintenberger \cite{fontaine-wintenberger, wintenberger} and a generalization of the
slope theory for isocrystals over the Robba ring introduced in \cite{kedlaya-revisited}.
We take a different approach that applies to arbitrary weights, by reformulating in terms of the universal filtration over the partial flag variety. This requires working in local coordinates on the flag variety;
the fact that the final construction does not depend on this choice (and so glues)
will follow because the universal crystalline local system is determined by its
fibres over rigid analytic points (see Definition~\ref{D:admissible2}).

\begin{defn}
Let $L$ be a field of characteristic $p$ equipped with a valuation $v_L$.
Let $L'$ be the completion of $L^{\perf}$ for the unique extension of $v_L$.
For $r>0$, let $\tilde{\calR}^{\bd,r}_L$ be the subring of $W(L')[p^{-1}]$ consisting of
$x = \sum_{i=m}^\infty p^i [x_i]$ for which $i + r v_{L}(x_i) \to +\infty$ as 
$i \to +\infty$. On $\tilde{\calR}^{\bd,r}_L$, the function $v_r(x) = \min_i \{i + r v_{L}(x_i)\}$
defines a valuation (that is, $e^{-v_r(\cdot)}$ is a multiplicative norm). 
Put $\tilde{\calR}^{\bd}_L = \cup_{r>0} \tilde{\calR}^{\bd,r}_L$; this is a field which is henselian (but not
complete) for the $p$-adic valuation.

Let $\tilde{\calR}^r_L$ be the Fr\'echet completion of $\tilde{\calR}^{\bd,r}_L$ for 
the $v_s$ for all $s \in (0,r]$. Put $\tilde{\calR}_L = \cup_{r>0} \tilde{\calR}^r_L$;
the Witt vector Frobenius $\phi$ on $W(L')$ extends continuously to $\tilde{\calR}_L$.
One defines isocrystals over $\tilde{\calR}_L$, and the associated notions of degree, slope, and \'etaleness,
by analogy with $\calR$. The analogue of Theorem~\ref{T:slope} carries over; see 
\cite[Corollary~6.4.3]{kedlaya-revisited}. One has an 
analogue of the Dieudonn\'e-Manin classification
theorem: if $L$ is algebraically closed, then
any \'etale isocrystal admits a $\phi$-invariant basis \cite[Theorem~4.5.7]{kedlaya-revisited}.
\end{defn}

We now make a relative analogue of the construction of Theorem~\ref{T:Berger}, working in local coordinates
on the partial flag variety $\calF_{D,H}$.
\begin{defn}
Let $S$ be the completion of $K_0[T_1^{\pm}, \dots, T_d^{\pm}]$ for the Gauss norm, for $d = \dim \calF_{D,H}$.
Let $\gotho_S$ be the subring of $S$ of elements of norm at most 1.
Extend the Witt vector Frobenius $\phi$ from $K_0$ to $S$ continuously so that $\phi(T_i) = T_i^p$
for $i=1,\dots,d$. 
Let $X$ be the open unit disc over $\calM(S)$, i.e., the set of $\alpha \in \calM(S[\pi])$ (for the Gauss norm
on $S[\pi]$) for which $\alpha(\pi) < 1$. We may also identify $X$ with the set of $\alpha \in \calM(\gotho_S \llbracket \pi
\rrbracket[p^{-1}])$ (for the Gauss norm on $\gotho_S \llbracket \pi \rrbracket$) for which $\alpha(\pi) < 1$.

The analogue of the commutative diagram 
\eqref{eq:base change} is
\begin{equation} \label{eq:base change2}
\xymatrix{
\gotho_S \llbracket \pi \rrbracket \ar^{\phi}[r] \ar^{\theta_n}[d] &
\gotho_S \llbracket \pi \rrbracket \ar^{\theta_{n+1}}[d] \\
\gotho_S(\epsilon_n) \llbracket t \rrbracket \ar[r] &
\gotho_S(\epsilon_{n+1}) \llbracket t \rrbracket
}
\end{equation}
in which the bottom horizontal arrow acts on $\gotho_S$ as $\phi$, fixes $\epsilon_n$, and carries $t$ to $pt$.
Define the group $\tilde{\Gamma} \cong \Gamma \ltimes \ZZ_p^d$
of automorphisms of $X$ in which $\Gamma$ acts as usual on $\calR$ and trivially on $S$,
while $(e_1,\dots,e_d) \in \ZZ_p^d$ acts as the $\calR$-linear substitution sending
$T_i$ to $(1 + \pi)^{e_i} T_i$ for $i=1,\dots,d$.

Choose an embedding of $\calM(S)$ into $\calF_{D,H}^{\an}$, as per Lemma~\ref{L:cover by discs}.
Let $\calE$ be the pullback of $D$ along the structural morphism $X \to \calM(K_0)$. 
Then there exists a coherent locally free sheaf $\calE'$ equipped with an isomorphism
$\calE[t^{-1}] \cong \calE'[t^{-1}]$, such that
for each positive integer $n$, the $t$-adic filtration on $\calE' \otimes_{\theta_n} S(\epsilon_n)(( t ))$ 
equals the $t$-adic filtration on $\calE \otimes_{\theta_n} S(\epsilon_n)((t))$
tensored with the filtration provided by the universal filtration over $\calF_{D,H}$.
(Beware that in this construction, $S(\epsilon_n)$ is to be viewed as an $S$-algebra
via $\phi^n$.)
More explicitly, we may obtain $\calE'$ by first modifying $\calE$ along $\pi = 0$, then pulling back by $\phi^n$
to obtain the appropriate modification along $\phi^n(\pi) = 0$.
The local freeness of $\calE'$ can be checked most easily by covering $\calF_{D,H}$ with the variety parametrizing partial flags
with marked basis, on which the verification becomes trivial.

One does not get an action of $\phi$ on $\calE'$ over all of $X$, because of poles introduced at $\pi=0$.
However, one does have an isomorphism $\phi^* \calE' \cong \calE'$ away from $\pi = 0$.
\end{defn}

\begin{defn}
Equip $\overline{S} = \Fp^{\alg}[T_1^{\pm}, \dots, T_d^{\pm}]$ and
$\overline{S}' = \overline{S}\llbracket \overline{z}\rrbracket$
with the trivial norm.
Then consider the diagram
\begin{equation} \label{eq:base change3}
\xymatrix{
\gotho_S \llbracket \pi \rrbracket \ar^{\phi}[r] \ar^{\theta_0}[d] &
\gotho_S \llbracket \pi \rrbracket \ar^{\theta_{1}}[d] \ar^{\phi}[r] & \cdots \\
\gotho_S(\epsilon_0) \ar^{\phi}[r] &
\gotho_S(\epsilon_{1}) \ar^{\phi}[r] & \cdots
}
\end{equation}
obtained from \eqref{eq:base change2}. Taking the completed direct limit over the top row
gives a map $\gotho_S \llbracket \pi \rrbracket \to W(\overline{S}^{\prime, \perf})$ sending $\pi$ to $[\overline{\pi}+1]-1$
and $T_i$ to $[\overline{T}_i]$; this map restricts to a map $\gotho_S \to W(\overline{S}^{\perf})$.
By Lemma~\ref{L:surjective}, the induced maps $\calM(W(\overline{S}^{\perf})) \to \calM(\gotho_S)$
and $\calM(W(\overline{S}^{\prime,\perf})) \to \calM(\gotho_S\llbracket \pi \rrbracket)$
are surjective.
Taking the completed direct limit over the bottom
row gives a map from $\gotho_S$ to the completion of $W(\overline{S}^{\perf})(\epsilon)$; again by
Lemma~\ref{L:surjective},
the induced map $\calM(W(\overline{S}^{\perf})(\epsilon)) \to \calM(\gotho_S)$ is surjective.
In fact, its fibres are permuted transitively by the action of $\tilde{\Gamma}$ on $\gotho_S \llbracket \pi \rrbracket$.
\end{defn}

\begin{defn} \label{D:admissible}
Put $\omega = p^{-p/(p-1)}$.
Given $\tilde{\alpha} \in \calM(W(\overline{S}^{\perf})(\epsilon))$,
let $\beta$ be the image of $\tilde{\alpha}$ under the map
$\theta^*: \calM(W(\overline{S}^{\perf})(\epsilon)) \to \calM(W(\overline{S}^{\prime,\perf}))$
induced by the vertical arrows in \eqref{eq:base change3}.
Note that $\mu(\beta)(\overline{\pi}) = \omega$.

For $L = \calH(\mu(\beta))$, the composition $\gotho_S\llbracket \pi \rrbracket \to 
W(\overline{S}^{\prime,\perf}) \to \tilde{\calR}_L$
extends to series in $\pi$ convergent
in an open annulus with outer radius 1. It thus makes sense to form the base extension
$\calE' \otimes \tilde{\calR}_L$, which is finite free over $\tilde{\calR}_L$ 
\cite[Theorem~2.8.4]{kedlaya-revisited}.
We say that $\alpha \in \calM(S)$ 
is \emph{admissible} if $\calE' \otimes \tilde{\calR}_L$ is \'etale
for some (hence any, thanks to the $\tilde{\Gamma}$-action)
choice of $\tilde{\alpha} \in \calM(W(\overline{S}^{\perf})(\epsilon))$
lifting $\alpha$.
\end{defn}

\begin{theorem} \label{T:admissible open}
The set $\calM(S)^{\adm}$ of admissible points of $\calM(S)$ is an open subset of $\calM(S) \cap \calF^{\wa}_{D,H}$
having the same rigid analytic points.
\end{theorem}
\begin{proof}
Openness will follow from the construction of the universal crystalline local
system (Theorem~\ref{T:construct} and Theorem~\ref{T:admissible}). 
The proof of Theorem~\ref{T:Berger}
shows that an arbitrary admissible point must also be weakly admissible.

It remains to check that any weakly admissible
rigid analytic point $\alpha \in \calM(S)$ is admissible. Put $K = \calH(\alpha)$; the lifts $\tilde{\alpha}$ of $\alpha$ can be put in bijection
with the components of $K \otimes_{K_0} K_0(\epsilon)$. In particular, for a fixed choice of $\tilde{\alpha}$,
the stabilizer $\tilde{\Gamma}_K$ of $\tilde{\alpha}$ in $\tilde{\Gamma}$ is an open subgroup.
Let $\calS_K$ denote the closure of the image of 
$S[\pi^{\pm}]$ in $\tilde{\calR}_L$.
Recall that $\tilde{\Gamma}$ contains $\Zp^d$ as a normal subgroup;
one calculates (as in Lemma~\ref{L:frac Robba}) that the 
$(\Zp^d \cap \tilde{\Gamma}_K)$-invariants of
$\calS_K$ form a copy of $\bB^\dagger_{\rig,K}$.
By matching up copies of $D$,
we obtain a $(\phi, \Gamma_K)$-equivariant
isomorphism of the $(\Zp^d \cap \tilde{\Gamma}_K)$-invariant submodule of 
$\calE \otimes \calS_K$ with the module $M$ from
Theorem~\ref{T:Berger}.
Since $t$ is invariant under $\Zp^d$,
we also obtain a corresponding isomorphism of primed objects.
The claim then follows.
\end{proof}

\begin{remark}
In the case of Hodge-Tate weights in $\{0,1\}$, Hartl defined the admissible locus $\calF^{\adm}_{D,H}$
(using a different but equivalent method), 
and showed that it is open in $\calF^{\wa}_{D,H}$ \cite[Corollary~5.3]{hartl-mixed}. He also
exhibited some examples where the two spaces differ \cite[Example~5.4]{hartl-mixed}; such examples
have also been exhibited by Genestier and Lafforgue.
Subsequently, Hartl showed (using results of Faltings) that $\calF^{\adm}_{D,H}$
is the image of the Rapoport-Zink period morphism \cite[Theorem~3.5]{hartl-mixed2}.
\end{remark}

\section{The universal crystalline local system}

Having identified a suitable candidate for the admissible locus on $\calF^{\an}_{D,H}$, 
we are ready to construct the universal crystalline local system over it.
(An \emph{\'etale local system}  of a nonarchimedean analytic space can be viewed as a 
representation of the \'etale fundamental group.
See \cite{dejong} for a full development.)

\begin{defn} \label{D:rings}
For $T$ an open subset of $\calM(\overline{S}^{\prime,\perf})$ and $r > 0$,
let $\calR_S^r(T)$ be the 
Fr\'echet completion of $\gotho_S \llbracket \pi \rrbracket[\pi^{-1}, p^{-1}]$ with respect to the restrictions of the seminorms
$\lambda(\gamma^{\log_\omega \rho}) \in \calM(W(\overline{S}^{\prime,\perf}))$ 
for all $\gamma \in T$ and all $\rho$ with $-\log_p \rho \in (0,r]$.
Let $\calR_S(T)$ be the union of the $\calR_S^r(T)$ over all $r>0$.
\end{defn}

\begin{theorem} \label{T:construct}
Suppose $\gamma \in \calM(\overline{S}^{\prime,\perf})$ is such that
for $L = \calH(\gamma)$, $\calE' \otimes \tilde{\calR}_L$ is \'etale.
Then there exist an open neighborhood $T$ of $\gamma$ in $\calM(\overline{S}^{\prime,\perf})$
such that for each positive integer $n$, there exists a finite \'etale
extension of $\calR_S^r(T)$ for some $r>0$ over which $\calE'$ acquires a basis on which $\phi$
acts via a matrix whose difference from the identity has $p$-adic valuation at least $n$.
(Note that $T$ is chosen uniformly in $n$.)
\end{theorem}
\begin{proof}
This is a calculation following \cite[Proposition~1.7.2]{hartl-equi}, which in turn follows
\cite[Lemma~6.1.1]{kedlaya-revisited}.
\end{proof}

\begin{defn} \label{D:open}
Suppose $\alpha \in \calM(S)$ is admissible.
Define $\tilde{\alpha}, \beta$ as in Definition~\ref{D:admissible},
and apply Theorem~\ref{T:construct} with $\gamma = \mu(\beta)$.
Then 
$\mu^{-1}(T)$ is open in $\calM(W(\overline{S}^{\prime,\perf}))$ because $\mu$ is continuous
(Theorem~\ref{T:lifting}), so $U = (\theta^*)^{-1}(\mu^{-1}(T))$ is open in
$\calM(W(\overline{S}^{\perf})(\epsilon))$.
Let $U'$ be the union of the $\tilde{\Gamma}$-translates of $U$, which is again open
in $\calM(W(\overline{S}^{\perf})(\epsilon))$.
Then $U$ and $U'$ have the same image $V_0$ in $\calM(\gotho_S)$, but $U'$ is the full inverse image of $V_0$ in
$\calM(W(\overline{S}^{\perf})(\epsilon))$. Hence the complement of $V_0$ is the image of a closed and thus
compact set (namely the complement of $U'$), and so is compact and thus closed. We conclude that $V_0$ is open
in $\calM(\gotho_S)$, and $V = V_0 \cap \calM(S)$ is open in $\calM(S)$. Since $\beta \in \mu^{-1}(T)$, we also have
$\alpha \in V$.
\end{defn}

\begin{theorem} \label{T:admissible}
Suppose $\alpha \in \calM(S)$ is admissible. Define an open neighborhood $V$ of $\alpha$ in $\calM(S)$
as in Definition~\ref{D:open}. Then there exists a $\Qp$-local system over $V$ 
whose specialization to 
any rigid analytic point of $V$ may be identified with the crystalline Galois representation produced by
Definition~\ref{D:to Galois}.
\end{theorem}
\begin{proof}
Retain notation as in Theorem~\ref{T:construct} and Definition~\ref{D:open}, 
and choose a nonnegative integer $n$ with $r > 1/(p^{n-1}(p-1))$.
Let $\alpha' \in V$ be any point, choose $\tilde{\alpha}' \in U$ lifting $\alpha'$,
and put $\beta' = \theta^*(\alpha') \in \mu^{-1}(T)$
and $\gamma' = \mu(\beta') \in T$.
Then $\calR^r_S(T)$ admits the seminorm
\[
\lambda((\gamma')^{p^{-n}})
= \lambda(\mu(\beta')^{p^{-n}}) = (\lambda \circ \mu)((\phi^{-n})^*(\beta')),
\]
which dominates $(\phi^{-n})^*(\beta')$ by Theorem~\ref{T:lifting}. We conclude that the elements of 
$\calR^r_S(T)$ define analytic functions on a subspace of $X$ containing $(\phi^{n*})^{-1}(V)$,
under the identification of $\calM(S)$ with the subspace $\pi=\epsilon_n-1$ of $X$.
As in Definition~\ref{D:to Galois},
by considering $\phi$-invariant sections of $\calE'$ over finite \'etale extensions of $\calR^r_S(T)$,
we obtain a $\Qp$-local system over $(\phi^{n*})^{-1}(V)$.
By also keeping track of the action of $\tilde{\Gamma}$, we obtain descent data yielding a $\Qp$-local system
over $V$ itself. The compatibility at rigid analytic points follows from the proof of
Theorem~\ref{T:admissible open}.
\end{proof}

\begin{defn}  \label{D:admissible2}
In Theorem~\ref{T:admissible}, any point of $V$ is automatically admissible. 
It follows that there exists an open
subset $\calF^{\adm}_{D,H}$ of $\calF^{\an}_{D,H}$
such that 
$\calM(S)^{\adm} = \calM(S) \cap \calF^{\adm}_{D,H}$ for any embedding
of $\calM(S)$ into $\calF^{\an}_{D,H}$ as in Lemma~\ref{L:cover by discs}.
We call $\calF^{\adm}_{D,H}$ the \emph{admissible locus} of $\calF^{\an}_{D,H}$.

By an argument similar to the proof of Theorem~\ref{T:admissible}, one shows that the definition
of the admissible locus, and the construction of the local system, does not depend on the
choice of local coordinates. We may thus glue using a cover as in Lemma~\ref{L:cover by discs}
to produce a
$\Qp$-local system over $\calF^{\adm}_{D,H}$ specializing to the crystalline representations associated
to rigid analytic points. We call this the \emph{universal crystalline local system} on $\calF^{\adm}_{D,H}$.
\end{defn}

\section{Further remarks}

\begin{remark} \label{R:Kisin}
In some cases of Hodge-Tate weights equal to $\{0,1\}$, the space $\calF^{\adm}_{D,H}$ receives a period morphism
constructed by Rapoport-Zink \cite{rapoport-zink}
from the generic fibre of the universal deformation space associated to a suitable $p$-divisible group.
One expects (as in \cite[Remark~7.8]{hartl-mixed})
that the pullback of the universal crystalline local system is obtained by extension of scalars from
a $\Zp$-local system which computes the (integral) crystalline Dieudonn\'e module at each rigid analytic point.
This appears to follow from work of Faltings (manuscript in preparation).

It should be possible
to give an alternative proof, more in the spirit of this lecture,
using Kisin's variant of Berger's proof of the Colmez-Fontaine theorem \cite{kisin}.
In Kisin's approach, the role of the highly ramified Galois tower $K_0(\epsilon)$ is played by the
non-Galois Kummer tower $\cup_n K_0(p^{-p^n})$. Instead of $(\phi, \Gamma)$-modules, one ends up
with modules over $\gotho_{K_0} \llbracket u \rrbracket$ carrying an action of the Frobenius lift
$u \mapsto u^p$, with kernel killed by a power of a certain polynomial. These are particularly well
suited for studying \emph{integral} properties of crystalline representations; indeed, their definition is inspired
by a construction of Breuil \cite{breuil} introduced precisely to study such integral aspects (moduli of
finite flat group schemes and $p$-divisible groups). We expect that one can carry out a close analogue
of the construction described in this lecture using Kisin's modules.
\end{remark}

\begin{remark}
When one considers the cohomologies of smooth proper schemes over a $p$-adic field which are no longer
required to have good reduction, one must broaden the class of allowed Galois representations slightly.
The correct class is Fontaine's class of \emph{de Rham} representations, which coincides with the class
of \emph{potentially semistable} representations by a theorem of Berger \cite{berger-cst}.
On the side of de Rham cohomology, one must replace the category of isocrystals by the category of
$(\phi, N)$-modules with finite descent data.
That is, one specifies not only a Frobenius action but also a linear endomorphism $N$ 
satisfying $N\phi = p\phi N$. 
(Note that any such $N$ is necessarily nilpotent.) It should be possible to follow the model of this lecture
to construct a universal \emph{semistable} local system; 
one possible point of concern is that the parameter space replacing
the partial flag variety is no longer proper.
\end{remark}

\begin{remark}
The techniques of this paper fit into the general philosophy that one can study representations of
arithmetic fundamental groups of schemes of finite type over a $p$-adic field using $p$-adic analytic methods.
This attitude has been convincingly articulated by Faltings in several guises, such as his $p$-adic analogue of
the Simpson correspondence \cite{faltings-simpson}. It has subsequently been taken up by Andreatta and his
collaborators (Brinon, Iovi\c{t}a), who have made great strides in developing and applying a relative
theory of $(\phi, \Gamma)$-modules \cite{andreatta-brinon, andreatta-iovita}.
\end{remark}

\begin{remark}
The relative $p$-adic Hodge theory in this paper has been restricted to the case where one starts with
a fixed isocrystal and varies its filtration, corresponding geometrically to a deformation to characteristic $0$
of a fixed scheme in characteristic $p$. It is also of great interest to consider cases where the 
isocrystal itself may vary.

On the Galois side, families of Galois representations parametrized by a rigid analytic space occur quite frequently in the theory of $p$-adic modular forms, dating back to the work of Hida \cite{hida} on ordinary families,
and continuing in the work of Coleman and Mazur \cite{coleman-mazur} on the eigencurve. More recently, 
the study of families of representations has become central in the understanding of the $p$-adic local
Langlands correspondence, particularly for the group $\GL_2(\Qp)$ \cite{colmez-reps}.

A partial analogue of the $(\phi, \Gamma)$-module functor for representations in analytic families
has been introduced by Berger and Colmez \cite{berger-colmez}. Unfortunately, it is less than clear what the essential image of the functor is; see \cite{kedlaya-liu} for some discussion. Moreover, proper understanding of families
of $(\phi, \Gamma)$-modules is seriously hampered by the lack of a good theory of slopes of Frobenius modules in 
families; the situation at rigid analytic points is understood thanks to \cite{kedlaya-relative}, but 
at nonclassical points things are much more mysterious.
Nonetheless, the construction has proved useful in
the study of Selmer groups in families, as in the work of Bella\"\i che \cite{bellaiche} and Pottharst
\cite{pottharst, pottharst2}.

One can also make analogous considerations on the side of Breuil-Kisin modules, as in the work of
Pappas and Rapoport \cite{pappas-rapoport}. Again, the correspondence from modules back to Galois representations is
somewhat less transparent in families, so the moduli stack of Breuil-Kisin modules itself becomes the central object of study; on this stack, Pappas and Rapoport introduce an analogue of the Rapoport-Zink period morphism.
One can also introduce an analogue of the admissible locus, as in work of Hellmann (in preparation);
in this line of inquiry, there appears to be some advantage in replacing 
Berkovich's theory of nonarchimedean analytic spaces with Huber's
more flexible theory of \emph{adic spaces} \cite{huber}.
\end{remark}

\end{document}